\documentclass[a4paper, 11pt]{amsart}
\usepackage[mathscr]{euscript}
\usepackage{mathrsfs}
\usepackage{mathtools, amsmath,amsthm, amssymb,latexsym,amsfonts}
\usepackage{thmtools, thm-restate}
\usepackage{enumerate}

\usepackage{breqn}

\newtheorem{theorem}{Theorem }[section]

\newtheorem{corollary}[theorem]{Corollary}

\newtheorem{lemma}[theorem]{Lemma}

\newtheorem{proposition}[theorem]{Proposition}

\newtheorem{remark}[theorem]{Remark}

\newtheorem*{theorem*}{Theorem}
\newtheorem*{remark*}{Remark}

\newcommand{\ud}{\mathrm{d}}

\newcommand{\Rd}{\mathbb{R}^d}
\newcommand{\R}{\mathbb{R}}

\newcommand{\MFRd} {M_F(\mathbb{R}^d)}
\newcommand{\CFE} {M_{F,S}(E)}

\newcommand{\MFOE} {{\overset{\circ}{M}}_F(E)}
\newcommand{\MFORd} {{\overset{\circ}{M}}_F(\Rd)}

\newcommand{\LL} {log-Laplace}

\newcommand{\SBM} {super-Brownian motion}
\newcommand{\intfour}[4]{\ensuremath{\int_{#1}^{#2} \big(S_{#3}v^\gamma(s)\big)(#4)\ud s}}

\newcommand{\OTS}{\left(\Omega\cap \{t<T(1)\}\right)}
\newcommand{\OTSRd}{\left(\Omega\cap \{t<T(1)\}\right)\times \mathbb{R}^d}

\newcommand{\OTG}{\left(\Omega\cap \{t<T(1)\},\mathcal{F}\cap \{t<T(1)\} \right)\times \left(\Rd,\mathcal{B}(\Rd)\right)}

\newcommand{\OTSN}{\Omega\cap \{t<T(1)\}\times \mathbb{R}^{N\times d}}

\newcommand{\FUNC} {\phi(z_1,z_2,\ldots,z_N)}

\newcommand{\DZ} {\ud z_1 \ud z_2\ldots \ud z_N}

\newcommand{\Xt} {\{X_t\}_{t\geq 0}}

\newcommand{\eqdist} { \overset{d}{=}}
\newcommand{\weakconv} {\overset{w}{\Longrightarrow} }

\numberwithin{equation}{section}

\begin{document}

\title[Super-Brownian motion with infinite mean]{Absolute continuity of the Super-Brownian motion with infinite mean}\
\author{Rustam Mamin}
\address{Faculty of Mathematics, Technion---Israel Institute of Technology, Haifa 32000, Israel}
\email{rst@technion.ac.il}
\author{Leonid Mytnik}
\address{Faculty of Industrial 
Engineering \& Management, Technion---Israel Institute of Technology, Haifa 32000, Israel}
\email{leonid@ie.technion.ac.il}
\date{December 2020}
\thanks{LM is supported in part by ISF grant No. ISF 1704/18.}
\begin{abstract}
In this work  we prove that for any dimension $d\geq 1$  and any  $\gamma \in (0,1)$ super-Brownian motion corresponding to the  \LL\ equation
\begin{equation*}
\begin{split}
\frac{\partial v(t,x)}{\partial t } & = \frac{1}{2}\bigtriangleup v(t,x) + v^\gamma (t,x) ,\: (t,x) \in \mathbb{R}_+\times\Rd,\\
v(0,x)&= f(x)
\end{split}
\end{equation*}
is absolutely continuous with respect to the  Lebesgue measure at any fixed time $t>0$.
Our proof is based on properties of solutions of the \LL\ equation.
 We also prove that when initial datum $v(0,\cdot)$ is a finite, non-zero measure,   then the  \LL\ equation  has a unique, continuous solution.
Moreover  this solution  continuously depends on initial data.
\end{abstract}

\maketitle

\clearpage

\vspace*{1\baselineskip}

\section{Introduction and main result}
This paper is devoted to studying regularity properties of the super-Brownian motion with stable branching mechanism with infinite mean.

Let us start with a notation. For a measure $\mu$ on $\Rd$ and a function $f$ on $\Rd$ let  $\langle \mu,f\rangle$ or $\langle f,\mu\rangle $ 
denote the  integral of a function $f$ with respect to a  measure $\mu$ (whenever it is well defined):
\begin{equation*}
\langle f,\mu \rangle= \langle\mu, f\rangle\equiv \int_{\Rd} f(x)\mu(\ud x).
\end{equation*}
Let $\gamma\in (0,2]\setminus \{1\}$. The super-Brownian motion with $\gamma$-stable branching mechanism,  $X=\{X_t, t\geq 0\}$,  is a Markov  measure-valued process on $\Rd$ 
which is characterized
as follows: for any finite measure $\mu$ and a nonnegative not identically zero function~$f\in L^{\infty}(\Rd)$,  
\begin{equation}
E_\mu\left(e^{-\langle X_t, \phi\rangle}\right)=
E\left( \left. e^{-\langle X_t, \phi\rangle}\right| X_0=\mu \right)= e^{-\langle \mu, v(t,\cdot)\rangle}, \quad \forall t\geq 0. 
\end{equation} 
Here $v$ is a solution to the so-called log-Laplace equation:
\begin{equation}
\label{eq:loglap1}
v(t,x) = \left(S_t f\right)(x) - \int_0^t (S_{t-s}v^\gamma(s,\cdot))(x)\ud s, \quad (t,x)\in \R_+\times\Rd,
\end{equation}
if $\gamma\in (1,2]$, and  for $\gamma\in(0,1)$, the sign in front of the non-linear term is reversed:
\begin{equation}
\label{eq:loglap2}
v(t,x) = \left(S_t f\right)(x) + \int_0^t (S_{t-s}v^\gamma(s,\cdot))(x)\ud s, \quad (t,x)\in \R_+\times\Rd. 
\end{equation}
Here and for  the rest of the paper  $\{S_t\}_{ t\geq 0}$ denotes a transition semigroup of Brownian motion 
whose generator is Laplacian $\frac{1}{2}\Delta$ in  $\Rd$. Clearly 
$$ S_t f(x) = \int_{\Rd} f(y)p_t(x-y)\,dy,\quad t\geq 0,$$ 
where   $\{p_t(x), t\geq 0, x\in \Rd\}$ is the transition density of the Brownian motion. 
 $\{S_t\}_{t\geq 0}$ describes the underlying Brownian motion of $X$,  whereas its
continuous-state branching mechanism described by $v\mapsto \pm v^\gamma, v\geq 0$. 
These equations were  considered by Watanabe in~\cite{Wat1} for a more general ``motion'' operator. 
As for the case of super-Brownian motion  with $\gamma$-stable branching mechanism (in what follows we will call it $\gamma$-super-Brownian motion)
it is well known that for $\gamma\in (1,2]$ in dimensions $d<\frac{2}{\gamma-1}$  at any fixed time $T>0$,  the measure
$X_t=X_t(dx)$ is absolutely continuous with probability one (cf.~\cite{bib:F88}). By an abuse of
notation, we sometimes denote a version of the density function of the measure $X_t=X_t(dx)$ by the same symbol, $X_t(dx)=X_t(x)\,dx$. 
It is even known that for $d=1$, $\gamma\in (1,2]$, at fixed times $t$,  there is a continuous version of the density in $x$ variable (see~\cite{MP}), and for $\gamma=2$, and again $d=1$,
there exists even jointly space-time continuous version of the density (see~\cite{KS}, \cite{bib:R89}). 
More detailed regularity properties of the densities of superprocesses with stable branching mechanism with possibly more general motion have been studied in~\cite{FMW1}, \cite{FMW2}, \cite{bib:MW15}, \cite{bib:MW16}. 

This paper is devoted to deriving absolute continuity of $X$ for the case of $\gamma\in (0,1)$. It is easy to check  that in this case $E(\langle X_t, 1\rangle)=\infty$, for $t>0$,  which adds some technical 
dificulties for the proofs.

Before we state the main result of this paper we need to introduce some notation.
\medskip
Let $E$ be any Polish space. 
Let  $C(E)$ and $B(E)$ be respectively  the   spaces of continuous and \emph{Borel} measurable  functions on space $E$. 
If $F(E)$ is a space of real-valued functions on $E$ we  define the  following subspaces of $F(E)$. 
$F_b(E)$ (respectively $F^+(E), F_c(E),  F_{bc}(E)$)  denotes the subspace of bounded (respectively positive, with compact support, bounded with compact support) functions.
For example,  $B^{+}_{bc}(\Rd)$ denotes a set of positive, bounded,  Borel
measurable functions with compact support on $\Rd$.

\medskip 

Now let us define the explosion time of the superprocess.

\begin{restatable}[Time of explosion]{exp_times_def}{exptimesrep}\label{def:exp_time}
Let  $d\geq 1$ and let  $\{X_t\}_{t\geq 0}$ be a super-Brownian motion  with non-random initial state $X_0$.
Given nonnegative continuous function $f$ on $\Rd$,  we define the  time of explosion 
$T(X_0,f)$ of $X$ as follows
\begin{equation*}
T(X_0,f)\equiv  \inf\left\{t\geq 0: \bigcap_{n=1}^\infty \left\{\langle X_t,f \rangle\geq n\right\} \right\}.
\end{equation*}
\end{restatable}
Now we  are able to state the main result of the paper.
\begin{restatable}[Absolute continuity]{main_result}{mainresultrep}\label{thm:abs_cont}
Let $d\geq 1$ and $0<\gamma<1$. Let $\Xt$ be a
$\gamma$-super-Brownian motion  with non-random initial state $X_0$ being a finite measure on $\Rd$.  For each $t>0$, $X_t(\ud x)$ is $P-a.s.$ absolutely continuous  on the   event $\{t<T(X_0,1)\}$.
\end{restatable}
The proof of this theorem is based on the properties of solutions  to the \LL\ equation corresponding to the process $\Xt$.
These properties are stated in  Theorem \ref {thm:semi}. This theorem extends results of Aguirre and Escobedo (see \cite{AgEs}) for the case of non-zero measure-valued initial conditions. The entire Section \ref{chap:semi}
is devoted to the proof  of these properties.
In Section \ref{sec:exp_times} we will  prove that for any  nonnegative, non-zero  continuous function $f$ on $\Rd$,
\begin{align*}
T(X_0,f) = T(X_0,1),\: P-\mathrm{a.s.}
\end{align*}
This property  allows us to define the  density of the superprocess $\Xt$ for a fixed time $t>0$.
In Section  \ref{sec:abs_cont} we conclude the  proof of   Theorem~\ref{thm:abs_cont} --- the main result of the paper.

\section{Semilinear heat equation}
For the rest of the paper fix $\gamma\in (0,1)$ and arbitrary dimension $d\geq 1$.  
One of the main tools   for investigating the $\gamma$-super-Brownian motion  is the log-Laplace equation
\par
\begin{equation}\label{eq:ll_eq_int}
v(t,x) = \left(S_t f\right)(x) + \int_0^t (S_{t-s}v^\gamma(s,\cdot))(x)\ud s.
\end{equation}
Usually in the literature~\eqref{eq:ll_eq_int} is studied for $f$ being a non-negative function. In the sequel we will consider~\eqref{eq:ll_eq_int} also with $f$ being a measure.   
\par
Before we discuss properties of~\eqref{eq:ll_eq_int} we need to introduce some notation. 
For a topological space $S$,  $\mathcal{B}(S)$ will denote the $Borel$ $\sigma\textrm{-algebra}$ on the  space~$S$.

We denote by $L^{p,w}(\Rd) $ (for $p=1$ or $p =\infty$) \label{def:L_ap} \label{def:l_a_p}a Banach  space of (equivalence classes of) measurable functions  on $\mathbb{R}^d$ with the norms:
\begin{align*}
\|f\|_{1,w}&\equiv \int_{\Rd} |w(x)f(x)| \ud x,  \: \:\textnormal{for}\:\: p=1\\
\left\|f\right\|_{\infty,\omega} &\equiv \inf \{M: Leb(x:|w(x)f(x)|>M)=0\},\:
\textnormal{for} \:p = \infty,
\end{align*}
\label{def:norm_l_ap}
where
\begin{align*}
w(x)\equiv C_w e^{-|x|},\:\:\int_{\Rd} w(x)\ud x = 1,
\end{align*}
and $Leb$ denotes the Lebesgue measure on $\Rd$.
$L^{p,w}_+(\Rd)$ (respectively $L^{p}_+(\Rd)$)  will denote the nonnegative elements of $L^{p,w}(\Rd)$ (respectively $L^{p}(\Rd)$). 

Given $L^{p,w}(\Rd)$ (for $p=1$ or $p =\infty$)  we define the  Banach  space 
$L^\infty_{loc}((0,\infty),L^{p,w}(\Rd))$   of (equivalent classes of)  measurable functions on $(0,\infty)\times \Rd$  as follows:
\vspace*{8pt}
$f \in L^\infty_{loc}\left((0,\infty),L^{p,w}(\Rd)\right)$ if and only if
 $f(t,\cdot) \in L^{p,w}(\Rd)$ for any fixed $t$ and
 \begin{equation*}
t\to \|f(t)\|_{p,w}\in L^\infty([a,b])
\end{equation*}
for any compact  interval $[a,b]\in (0,\infty)$. Similarly $L^\infty_{loc}((0,\infty),L^{p,w}_+(\Rd))$ is defined.

By $M_F(E)$ (respectively $\CFE$) we  denote the space  of finite measures and  finite   (respectively finite signed) signed measures on  a Polish space $E$ equipped with the topology of the weak convergence. We write
\begin{equation*}
\mu_n \weakconv \mu,\quad \text{as $n\rightarrow\infty$,}
\end{equation*}
if  the sequence $\{\mu\}_{n=1}^\infty$ of finite measures or finite signed measures  weakly converges to a finite measure $\mu$.

If $F$ is a set of functions or  measures then $\overset{\circ}{F}$ \label{def_circ} denotes this set without zero element, that is $\overset{\circ}{F}= F\setminus \{0\}$. If $F$ is a topological space then, topology of  $\overset{\circ}{F}$ is inherited from $F$.
For example $\MFOE$  is a space of finite  non-zero measures on Polish space $E$ with the topology inherited from $M_F(E)$.

With all this notation at hand we can get back to~ \eqref{eq:ll_eq_int}. 
Equation \eqref{eq:ll_eq_int} was studied by  Aguirre and Escobedo in~\cite{AgEs}).
We state some of their results in the  following  theorem.
\begin{theorem}[Aguirre and Escobedo \cite{AgEs}]\label{thm:ag_es_res}
For any not identically zero   $f \in L^{\infty,w}_+(\Rd)$ there exists the unique   solution $v(t,x,f)$ of  equation \eqref{eq:ll_eq_int} such that
\begin{itemize}
\item[$(1)$]
$v(\cdot,\cdot,f) \in C^+\left((0,\infty)\times \Rd\right)\cap L^\infty_{loc}\left((0,\infty),L^{\infty,w}_+(\Rd)\right)$;

\item[$(2)$]

$v(t,x,f)> ((1-\gamma)t)^{1/(1-\gamma)},\:(t,x) \in (0,\infty)\times \Rd$,

\item[$(3$]

for $i=1,2$, let $v(t,x,f_i)$ be the solution to \eqref{eq:ll_eq_int} with initial condition $v(0,\cdot,f_i)=f_i$.
If $f_1(x)\leq f_2(x),\: a.e.\: x$,  then
\begin{align*}
v(t,x,f_1) \leq v(t,x,f_2),\: \forall (t,x) \in (0,\infty)\times \Rd.
\end{align*}
\item[$(4)$] $\lim_{t\to 0} v(t,\cdot,f) = f,\: \textnormal{for}\: a.e. \: x\in \Rd$;

\item[$(5)$]

for any fixed $(t,x)\in (0,\infty)\times \Rd$ the mapping
\begin{equation*}
v(t,x,\cdot): L^{\infty,w}_+(\Rd) \mapsto \mathbb{R}_{++}
\end{equation*}
is continuous.
Here $\mathbb{R}_{++}\equiv (0,\infty).$
\end{itemize}
\end{theorem}
\begin{remark*}
In fact, Aguirre and Escobedo prove the  above theorem for a more general class of initial data.
\end{remark*}
We extend the results in  Theorem \ref{thm:semi} for the case of not identically zero measure-valued initial conditions.
Consider the following equation:
\begin{equation}\label{eq:ll_eq}
v(t,x) = (S_t\mu)(x) + \int_0^t (S_{t-s} v^\gamma(\cdot,s))\ud s, \: x \in \Rd,\: t>0,
\end{equation}
where $\mu \in {\MFORd}$, and again $d\geq 1$ is an arbitrary dimension. We set 
$S_t\mu(x)=\int_\Rd p_t(x-y)\mu(dy), x\in \Rd$. 
In order to stress dependence of the  solutions of this equation on initial data we will sometimes write
$v(t,x,\mu)$. In what follows we will also use the following notation for solutions of \eqref{eq:ll_eq}:
\begin{equation}
V_t(\mu)(x)\equiv v(t,x,\mu),\quad t>0, x\in \Rd, 
\end{equation}
for $\mu\in \MFORd$ or being a non-negative, not identically zero function.

\par
Before we state the main result of this section, 
 let us  define  the constant $\gamma'$ in terms of $\gamma $ as follows:
\begin{equation*}
\gamma' = \frac{1}{1-\gamma}.
\end{equation*}

\begin{theorem}[Existence, uniqueness and dependence on initial data]\label{thm:semi}
 For \par any $\mu \in \MFORd$,   equation \eqref{eq:ll_eq} has the unique solution $v(t,x)$ such that
\begin{equation*}
v(\cdot,\cdot,\mu) \in L_{loc} ^\infty\left((0,\infty) , L^{1,w}_+(\mathbb{R}^d)\right)\cap C^+((0,\infty)\times \Rd)
\end{equation*}
and
\begin{equation}\label{eq:sol_ineq}
\left((1-\gamma)t \right)^{\gamma'}<v(t,x,\mu)\leq e^t\left(S_t\mu \right)(x) + e^t,\: 0<t<\infty.
\end{equation}
Moreover, this solution continuously depends on initial data:
if a sequence $\{\mu_n\}_{n=1}^\infty$  from $\MFORd$  converges weakly to $\mu\in \MFORd$ then
\begin{equation*}
 \lim_{n\to \infty } v(t,x,\mu_n)\to v(t,x,\mu),
\end{equation*}
for any $(t,x) \in (0,\infty)\times \Rd$.
\end{theorem}
\begin{remark}
Since for any $t \in (0,\infty)$,
\begin{align*}
e^t\left(S_t\mu \right)(x) + e^t \in L^{1,w}(\mathbb{R}^d),
\end{align*}
it easily follows  from   inequality \eqref{eq:sol_ineq}  that the sequence of solutions\par $\{v(\cdot,\cdot,\mu_n)\}_{n=1}^\infty$ which also converges to $v(\cdot,\cdot,\mu)$ in $L_{loc} ^\infty\left(((0,\infty) , L^{1,w}_+(\mathbb{R}^d)\right)\cap C^+((0,\infty)\times \Rd)$.
\end{remark}
 The proof of the next lemma is trivial and hence it is omitted.
\begin{lemma}\label{lm:simp_ineq}
Let $\mu\in\MFRd$.  Then, for any $t\in (0,\infty)$,
\begin{equation*}
\left(S_t\mu\right)(x) \leq \frac{\mu(\Rd)}{(2\pi t)^{d/2}},\: \forall x \in \Rd.
\end{equation*}
\end{lemma}
Now we are ready to state the corollary to  Theorem \ref{thm:semi}.
\begin{corollary}\label{cor:conv_bounded_mesure}
Let $\{\mu_n\}_{n=1}^\infty \subset \MFORd$ be a sequence of measures that   converges weakly to $\mu \in \MFORd$ and let $v(\cdot,\cdot,\mu_n)$ be the corresponding solutions of  \eqref{eq:ll_eq}. Then, for any $\chi \in \MFRd$ and $t\in (0,\infty)$,
\begin{equation}\label{eq:conv_bound_measure}
\lim_{n\to \infty} \int_{\Rd} v(t,x,\mu_n)\chi(\ud x) = \int_{\Rd} v(t,x,\mu)\chi(\ud x).
\end{equation}
\end{corollary}
\begin{proof} 
By Theorem \ref{thm:semi} we have
\begin{equation*}
v(t,x,\mu)\leq e^t\left(S_t\mu \right)(x) + e^t,n=1,2,\ldots,\: 0<t<\infty.
\end{equation*}
Since the sequence $\{\mu_n\}_{n=1}^\infty$  converges weakly to $\mu$, then by Lemma \ref{lm:simp_ineq}
\begin{equation*}
\begin{split}
\sup_{x \in \Rd} \left(e^t\left(S_t\mu\right)(x) + e^t\right)\leq
e^t\left( \frac{1}{(2\pi t)^{d/2}} \sup_{n\geq 1}\mu_n(\Rd) + 1  \right)< \infty.
\end{split}
\end{equation*}
Thus,  we conclude that the sequence
$\{v(t,\cdot,\mu_n)\}_{n=1}^\infty$ is bounded.
Also by  Theorem \ref{thm:semi} $\{v(t,\cdot,\mu_n)\}_{n=1}^\infty$ converges pointwise to $v(t,\cdot,\mu)$.
 Hence,  by the  bounded convergence theorem,  \eqref{eq:conv_bound_measure} follows.
\end{proof}

Theorem~\ref{thm:semi} will be proved in Section~\ref{chap:semi}.

\section{Proof of Theorem~\ref{thm:abs_cont}}\label{chap:abs_cont}
In this section we prove the main result of this paper --- absolute continuity of the \SBM\  $X$ with the branching mechanism $v\mapsto v^\gamma$, for $\gamma \in (0,1)$.  

In~Section~\ref{sec:exp_times} we investigate explosion time for the $\gamma$-super-Brownian: this is necessary for the proof of  Theorem~\ref{thm:abs_cont} that will be concluded in
Section~\ref{sec:abs_cont}.  
\subsection{Explosion times}\label{sec:exp_times}
As we will see for any $t>0$, the  $\gamma$-\SBM\  $X=\Xt$  explodes  by time  $t$ with non-zero probability.
In this section we  investigate the  distribution  of the explosion times.
We assume that  $X$ is defined on probability space $(\Omega, P, \mathcal{F}, \mathcal{F}_t)$ and adapted to  filtration
$\{\mathcal{F}_t\}_{t\geq 0}$. We also assume that the   initial state $X_0$ of $X$  is a non-random finite measure.
\begin{remark}\label{rm:strong_markov}
By Corollary $4.3.2$, in~\cite{Daw1},  it is easy to show  that $\Xt$ is a Feller process and therefore it  has a strong Markov property.
\end{remark}

The next lemma states the elementary properties of the explosion times.
The proofs are simple and easily follow  from the definition, so they are omitted.
\begin{lemma}\label{lm:expl_times_1}
\begin{enumerate}[(1)]
\item For any function $f\in  \overset{\circ} {L}\mbox{}^{\infty}_+(\Rd)$, and any  $a\in (0,\infty)$,
\begin{align*}
T(X_0,f) = T(X_0, af),\: P-\textnormal{a.s.}
\end{align*}
 \item For any $f\in   \overset{\circ} {L}\mbox{}^{\infty}_+(\Rd)$,
 \begin{equation*}
 T(X_0,1)\leq T(X_0,f) ,\: P-\textnormal{a.s.}
 \end{equation*}
\end{enumerate}
\end{lemma}

In the next lemma we will show that for any $t\geq T(X_0,f)$, one has \mbox{$X_t(f)=\infty$}.
Before we proceed, let us recall 
 that the Laplace transform of  $\Xt$ is given by 
\begin{equation}
\label{eq:expl_lapl1}
E\left( e^{-\langle X_t,f \rangle} \right)= e^{- \langle X_0,V_t(f) \rangle},\: f \in  \overset{\circ} {L}\mbox{}^{\infty}_+(\Rd),
\end{equation}
where $\{V_t(f)\}_{t\geq 0}$ solves  log-Laplace equation~\eqref{eq:ll_eq_int}. 
\begin{lemma}
For any $t>0$, $f\in \ \overset{\circ} {L}\mbox{}^{\infty}_+(\Rd)$,
\begin{equation*}
\left\{ T(X_0,f) \leq t\right\} = \left\{X_t(f)=\infty\right\}, \: P-\textnormal{a.s.}. 
\end{equation*}
\end{lemma}
\begin{proof}
The $P$-a.s. inclusion $\left\{X_t(f)=\infty\right\} \subset \{ T(X_0,f) \leq t \}$ is trivial. Now let us show $\{ T(X_0,f) \leq t\} \subset \{X_t(f)=\infty \}$, $P$-a.s.
It is enough to verify  that
\begin{equation}\label{eq:exp_time_4}
E\left( e^{- \langle X_t,f \rangle} 1_{\left\{ T(X_0,f) \leq t\right\}} \right)=0.
\end{equation}
Define the stopping  time
\begin{align*}
T_n(X_0,f) \equiv \inf\left\{ t \geq 0, X_t(f)=n\right\}.
\end{align*}
Clearly  $ T_n(X_0,f) \to  T(X_0,f),\: P$-a.s.,  as $n\to \infty$.
Then, for any  $\delta \in (0,t)$ arbitrarily small,
\begin{align}
E& \left( e^{ - \langle X_t,f \rangle}  1_{\left\{ T(X_0,f)\leq t-\delta\right\} } \right)=\notag \\
 &= E \left(   e^{-\langle X_t,f \rangle} 1_{\left\{ T(X_0,f)\leq t-\delta\right\} } 1_{\left\{ T_n(X_0,f)\leq t-\delta\right\} }   \right)\notag \\
&\leq E \left( e^{ - \langle X_t,f \rangle}  1_{\left\{ T_n(X_0,f)\leq t-\delta\right\} } \right)\notag\\
&=E\left( E \left(e^{-\langle X_t,f \rangle}  | \mathcal{F}_{T_n(X_0,f)} \right) 1_{\left\{ T_n(X_0,f)\leq t-\delta\right\}} \right)  \label{eq:lm_stop_time_1}\\
&= E\left(  e^{ -\langle X_{T_n(X_0,f)},V_{t-T_n(X_0,f)}(f) \rangle }1_{\left\{ T_n(X_0,f)\leq t-\delta \right\}}\right).\label{eq:lm_stop_time_2}
\end{align}
Here, in \eqref{eq:lm_stop_time_1},  we used the strong  Markov Property (see Remark \ref{rm:strong_markov}).
Fix $c(\delta)>0$ sufficiently small such that $c(\delta) f(x)\leq  ((1-\gamma)t)^{\gamma'}$ for all \mbox{$t\geq \delta$}, \mbox{$ x\in \Rd$}. Then by Theorem~\ref{thm:ag_es_res} (see also  
  Lemma $2.2$ in~\cite{AgEs}) we have 
\begin{align*}
c(\delta )f(x) \leq ((1-\gamma)t)^{\gamma'} \leq V_t(f)(x),\: \forall t\geq \delta, x\in \Rd. 
\end{align*}
Therefore the expression \eqref{eq:lm_stop_time_2} can be bounded from the above  by
\begin{align*}
E\left(e^{-c(\delta)n} 1_{\left\{ T_n(X_0,f)\leq t-\delta \right\}}\right).
\end{align*}
By the dominated convergence theorem this expression tends to zero, as $n\to \infty$  and we get
\begin{align*}
E\left( e^{-\langle X_t,f \rangle}1_{\left\{ T(X_0,f)\leq t-\delta \right\}}\right) =0, \forall t>0.
\end{align*}
Now take $\delta \searrow 0$ and by the monotone convergence theorem we get \eqref{eq:exp_time_4} and this completes  the proof.
\end{proof}
The following corollary is immediate. 
\begin{corollary}\label{cor:exp_times_1}
Let $f \in   \overset{\circ} {L}\mbox{}^{\infty}_+(\Rd)$.  Then
\begin{align}
 E \left(e^{-\langle X_t,f \rangle}\right) = E\left(e^{-\langle X_t,f \rangle}1_{\left\{ T(X_0,f)>t\right\} }\right),\: \forall t>0.
\end{align}
\end{corollary}
\medskip
Now, let us calculate the distribution of $T(X_0,1)$ --- the distribution of the explosion time of the total mass of the super-Brownian motion $X$.
By Corollary \ref{cor:exp_times_1} and \eqref{eq:expl_lapl1}, we get

\begin{equation}\label{eq:calc_tot_exp}
\begin{split}
P( t<T(1,X_0))& =  \lim_{a\searrow 0}E \left(1_{\{t<T(1,X_0)\}}\left( e^{-a\langle X_t,1 \rangle}\right)\right)\\
&=  \lim_{a\searrow 0} e^{-\langle V_t(a),X_0 \rangle}\\
& =  \lim_{a\searrow 0}\exp\left(-\langle X_0,1 \rangle\big(a^{1-\gamma}+t(1-\gamma)\big)^{\gamma'}\right)\\
&=  \exp\left( - \langle X_0,1 \rangle  \big( t(1-\gamma)\big)^{\gamma'}\right),
\end{split}
\end{equation}
where the third equality follows from the fact that $V_t(a)$ is a solution of the  ordinary differential  equation

\begin{equation*}
\left\{
\begin{array}{ll}
\displaystyle{\frac{\ud v(t)}{\ud t}} &= v^\gamma(t), \quad t\geq0, \\[ 10pt]
v(0) &= a,
\end{array} \right.
\end{equation*}
and hence
\begin{align}\label{eq:sol_ode}
V_t(a)(x) = ( a ^{1-\gamma} + t(1-\gamma))^{\gamma'},\quad \forall x\in\Rd, t\geq 0.  
\end{align}
Then
\begin{equation}\label{eq:expl_dist}
F_{T(1,X_0)}(t)= P(t\geq T)= 1 -\exp\left( - \langle X_0,1 \rangle \big( t(1-\gamma)\big)^{\gamma'}\right).
\end{equation}
But what about other test functions $f$? What is the law of  $\langle X_t,f \rangle$ for a general  $f\in   \overset{\circ} {L}\mbox{}^{\infty}_+(\Rd)$? The answer is given in the following lemma.
In what follows, in order to simplify  notation, we often write $T(f)$ instead of $T(f,X_0)$.
\begin{lemma}\label{lm:expl_times_2}
For any $f\in   \overset{\circ} {L}\mbox{}^{\infty}_+(\Rd)$ random variable $T(f)$ has the same distribution as $T(1)$:
\begin{equation*}
 F_{T(f)}(t)= P(t\geq T(f))= 1 -\exp\left( -  \langle X_0,1 \rangle \big( t(1-\gamma)\big)^{\gamma'}    \right), \quad t\geq 0. 
\end{equation*}
\end{lemma}
\begin{proof}
Fix an arbitrary  $f\in   \overset{\circ} {L}\mbox{}^{\infty}_+(\Rd)$, and $t>0$. Then  we have

\begin{equation}\label{eq:exp_time_calc}
\begin{split}
P( t<T(f))& =  \lim_{a\searrow 0}E \left(1_{\{t<T(f)\}}e^{-\langle X_t,af \rangle}\right)\\
&= \lim_{a\searrow 0}E \left(e^{-\langle X_t,af \rangle}\right)\\
&=  \lim_{a\searrow 0} e^{-\langle V_t(af),X_0 \rangle},
\end{split}
\end{equation}
where the second equality follows by Corollary \ref{cor:exp_times_1}.
By Theorem \ref{thm:ag_es_res} (see also  Lemma $2.2$ in~\cite{AgEs}), we get
\begin{equation}\label{eq:lm_eq_1}
\left((1-\gamma)t\right)^{\gamma'} < 
V_t(af)(x),
\: \forall  a >0, t>0,  x\in \Rd.
\end{equation}
Using \eqref{eq:lm_eq_1}, 
we get
\begin{align}
\label{eq:15_12_1}
 \exp\left(-\left<X_0, V_t(af)\right>\right) 
&\leq \exp\left(-\left<X_0,1  \right> \left( t(1-\gamma)\right)^{\gamma'}  \right).
\end{align}
By Lemma~\ref{lm:expl_times_1}(2) we have $T(1)\leq T(f)$, $P$-a.s. 
By 
this,  \eqref{eq:expl_dist}, \eqref{eq:exp_time_calc}, and \eqref{eq:15_12_1} we  obtain
\begin{align*}
P( t<T(1))
&\leq  P( t<T(f))\\
&\leq \exp\left(-\left<X_0,1 \right> \left( t(1-\gamma)\right)^{\gamma'}   \right)\\
&=P( t<T(1)).
\end{align*}
Thus, we get $P( t<T(f))= P( t<T(1))$. Since $t>0$ was arbitrary, we are done. 
\end{proof}

The next lemma is a consequence of the first two lemmas in this section.
\begin{lemma}\label{lm:expl_times_3}
For any  $f\in   \overset{\circ} {L}\mbox{}^{\infty}_+(\Rd)$,
\begin{equation*}
P\left(T(1))\neq T(f)\right) = 0.
\end{equation*}
\end{lemma}
\begin{proof}
By Lemma \ref{lm:expl_times_1}(2), $T(1)\leq T(f)$,  $P$-a.s. and, by Lemma \ref{lm:expl_times_2}, $T(1)$ and $T(f)$ have the same distribution, hence the result follows.
\end{proof}
\begin{corollary}\label{cor:expl_time}
For any $f\in   \overset{\circ} {L}\mbox{}^{\infty}_+(\Rd)$,
\begin{align*}
E \left(  e^{-\langle X_t,f \rangle}  \right) = E \left( e^{-\langle X_t,f \rangle} 1_{\left\{t<T(1)\right\}}\right) = e^{-\langle X_0,V_t(f) \rangle},\;\; t>0. 
\end{align*}
\end{corollary}

\subsection{Proof of Theorem~\ref{thm:abs_cont}}\label{sec:abs_cont}
We begin this subsection with the following  remark.
\begin{remark}\label{rem:radon_nikodym}
By Lemma $3.4.2.1$ in~\cite{Daw2}, any random measure $Y \in \MFRd$ can be decomposed  into its absolutely continuous $Y^{ac}$ and singular $Y^{s}$ parts  with respect to the Lebesgue measure:
$Y(\omega,\ud x) = Y^{ac}(\omega,\ud x)+ Y^{s}(\omega,\ud x)$.
By the  definition of $T(1)$,  $X_t$ is  a finite measure on  $\Omega\cap \{t<T(1)\}$.
Hence on the set $\Omega\cap \{t<T\}$, $X_t$ can be decomposed into absolutely continuous and singular parts
\begin{align*}
X_t(\omega,\ud x) = X_t^{ac}(\omega,\ud x)+ X_t^{s}(\omega,\ud x).
\end{align*}
\end{remark}

\par
The next lemma is used in the proof of  measurability of density. Its proof is a simple  adaptation of the proof of Theorem $1.8$ from~\cite{Li} and therefore it is omitted.
\begin{lemma}\label{lm:meas}
For any $f\in B^+_{bc}(\Rd)$ and any fixed $t\in (0,\infty)$, the map
$(\omega,z) \mapsto \langle  X_t(\omega),f(z-\cdot)\rangle$ is a measurable map from \par ${\OTG}$   to $\mathbb{R}_+$.
\end{lemma}

For studying differentiability properties of  $X_t$, let us introduce a sequence of functions $\{\delta^n(\cdot)\}_{n=1}^\infty$ defined
 as
\begin{equation*}
\delta^n(x) = \left\{
\begin{array}{ll}
 1/Leb(B_{1/n}(0)),& \text{if } \left|x\right|\leq \frac{1}{n};\\ 
 0,& \text{otherwise.}
\end{array} \right.
\end{equation*}
Here $B_{1/n}(0)$ is a closed ball of radius $1/n$, centered at the origin.
Notice that  the sequence $\{\delta^n(z-\cdot)\}_{n=1}^\infty$   converges to Dirac $\delta$-function with support at  point $z$.\par
\begin{lemma}\label{lm:radon_nikodym}
On  $\OTSRd$, $P(\ud \omega)\ud z$ - a.e. there exists
a limit
\begin{equation*}
\widetilde{\eta}^{ac}_t(\omega,z)= \lim_{n\to \infty} \langle X_t(\omega),\delta^n(z-\cdot)\rangle.
\end{equation*}
The random function $\widetilde{\eta}^{ac}_t$ is a version of the Radon-Nikodym derivative of
$X_t$ on $\OTS$.
 Moreover $\widetilde{\eta}^{ac}_t$ is a measurable map from \par
 $\OTG$ to $\mathbb{R}_+$.
\end{lemma}
\begin{proof}
By the Lebesgue density theorem (see~\cite{Rudin},   Theorem $7.14$), for $P$-a.s. $\omega \in \OTS$, there exists a limit
\begin{equation}\label{eq:radon_nikodym}
\widetilde{\eta}^{ac}_t(\omega,z)= \lim_{n\to \infty} \langle X_t(\omega),\delta^n(z-\cdot)\rangle
\end{equation}
for all $z\in R^d\setminus N(\omega)$ where $N(\omega)$ is a Borel subset of the Lebesgue measure zero and $\widetilde{\eta}^{ac}_t$ is a Radon-Nikodym derivative with respect to the Lebesgue measure.
It is easy to see that convergence in \eqref{eq:radon_nikodym} takes place \mbox{$P(\ud \omega)\ud z$-a.e.}
We set $\widetilde{\eta}^{ac}_t(\omega,z)$ to be zero  at points $(\omega,z)$ where the limit does not exist.
\par
By Lemma~\ref{lm:meas},  for each $n=1,2,\ldots$, $\langle X_t(\omega),\delta^n(z-\cdot)\rangle$ is measurable and the
measurability of $\widetilde{\eta}^{ac}_t(\omega,z)$ follows from $P(\ud \omega)\ud z$-a.e. convergence.
\end{proof}
Function $\widetilde{\eta}^{ac}_t(\omega,z)$ is defined on $\OTS$. Function $\eta^{ac}_t(\omega,z)$ is an extension  of function $\widetilde{\eta}^{ac}_t(\omega,z)$ to entire $\Omega$:
\begin{equation}\label{eq:def_eta}
\eta^{ac}_t(\omega,z) = \left\{
\begin{array}{ll}
\widetilde{\eta}^{ac}_t(\omega,z) & \text{if } \omega\in  \OTS,\\
 \infty   & \text{otherwise. }
\end{array} \right.
\end{equation}

Recall that for any $\mu\in \MFORd$, $\{V_t(\mu)\}_{t>0}$ denotes the solution to~\eqref{eq:ll_eq}.  
\begin{lemma}\label{lm:ll_conv_as}
For every $t\in (0,\infty)$
the equality
\begin{equation*}
E \left(1_{\{t<T(1)\}} \exp\left(-\sum_{i=1}^N a_i\eta^{ac}(z_i)\right)\right) = \exp\left(- \left(\left\langle X_0, V_t\bigg(\sum_{i=0}^N a_i \delta(z_i- \cdot)\right)\right\rangle \right)
\end{equation*}
holds  for almost every $\{z_i\}_{i=1}^N \subset \mathbb{R}^d$ and any 
$\{a_i\}_{i=1}^N\subset \R_{++}\,$.
\end{lemma}
\begin{proof}

Let $\FUNC$ be any function in  $ C^+_b(\mathbb{R}^{d\times N})\cap L^1(\mathbb{R}^{d\times N})$.
By Corollary~\ref{cor:expl_time} we have
\begin{equation*}
E \left(1_{\{t<T(1)\}}e^{- \left< X_t,\sum_{i=1}^N a_i \delta^n(z_i-\cdot) \right> }\right)= e^{-\left<X_0,V_t\left(\sum_{i=1}^N a_i\delta^n(z_i-\cdot)\right) \right>}.
\end{equation*}
Let us multiply both parts of  this equation by the function $\FUNC$,
 integrate over  $\mathbb{R}^{d\times N}$ and take the limit
 \begin{equation}\label{eq:as_conv1}
 \begin{split}
 \lim_{n\to \infty}& \int_{\mathbb{R}^{d\times N}} E\left(1_{\{t<T(1)\}} e^{-X_t(\sum_{i=1}^N a_i \delta^n (z_i-\cdot))}\right) \FUNC \DZ  \\
 &= \lim_{n\to \infty} \int_{\mathbb{R}^{d\times N}}e^{- \langle V_t(\sum_{i=1}^N a_i \delta^n(z_i-\cdot)),X_0 \rangle}\FUNC \DZ.
 \end{split}
 \end{equation}

 By Lemma~\ref{lm:radon_nikodym}, the limit

 \begin{equation*}
 \lim_{n\to \infty} \left<X_t(\omega),\sum_{i=1}^N a_i \delta^n(z_i-\cdot) \right> =
 \sum_{i=1}^N a_i\eta^{ac}_t(\omega,z_i)
 \end{equation*}
 exists almost everywhere  on $\OTSN$ with respect to the  measure $P(\ud \omega)\FUNC\DZ$.
 Therefore, by the bounded convergence theorem,  we get following limit on the left-hand side of
\eqref{eq:as_conv1}:

\begin{equation}\label{eq:as_conv2}
 \begin{split}
& \lim_{n\to \infty} \int_{\mathbb{R}^{d\times N}} E\left(1_{\{t<T(1)\}} e^{- \langle X_t,\sum_{i=1}^N a_i \delta^n(z_i-\cdot) \rangle}\right) \FUNC \DZ  \\
 &=\int_{\mathbb{R}^{d\times N}}  E\left(1_{\{t<T(1)\}} e^{-\sum_{i=1}^N a_i \eta^{ac}_t(z_i)}\right) \FUNC\DZ.
 \end{split}
 \end{equation}
Now let us take care of the right-hand side of \eqref{eq:as_conv1}.
 Since $X_0$ is a finite, non-random measure,  then  by Corollary  \ref{cor:conv_bounded_mesure},  the right hand side of  equation \eqref{eq:as_conv2} also converges:
 \begin{equation}\label{eq:as_conv3}
 \begin{split}
 \lim_{n\to \infty}& \int_{\mathbb{R}^{d\times N}}e^{-\langle V_t(\sum_{i=1}^N a_i \delta^n(z_i-\cdot)),X_0 \rangle} \FUNC\DZ \\
 &= \int_{\mathbb{R}^{d\times N}}e^{-\langle V_t(\sum_{i=1}^N a_i \eta^{ac}(z_i-\cdot)),X_0 \rangle}\FUNC\DZ .
 \end{split}
 \end{equation}
Now, since $\phi$ was chosen  arbitrarily,  we can  combine  \eqref{eq:as_conv1}, \eqref{eq:as_conv2}
and  \eqref{eq:as_conv3}  and get
\begin{equation*}
 \begin{split}
 E\left(1_{\{t<T(1)\}} e^{-\sum_{i=1}^N a_i\eta^{ac}_t(z_i)}\right) &=  e^{-\langle V_t(\sum_{i=1}^N a_i \delta(z_i-\cdot)),X_0 \rangle} \\
 & \text{for Lebesgue almost every} \{z_i\}_{i=1}^N \: \text{in}\:  \mathbb{R}^d.
 \end{split}
 \end{equation*}
 \end{proof}
\begin{lemma}\label{lm:as_conv}
Let $\phi\in C^+_b(\Rd)\cap L^1(\Rd)$
and let $\{\xi_n\}_{n=1}^\infty$ be a sequence of i.i.d. random variables defined on some probability space $(\Omega',\mathcal{F}',P')$ with  the probability density function
\begin{equation*}
g^r_{\xi}(x) = \left\{
\begin{array}{ll}
 1/Leb(B_r(0)) ,& \textnormal{if } |x| \leq r\\
 0,  & \textnormal{elsewhere. }
\end{array} \right.
\end{equation*}

Then,  for any $f \in L^1(\Rd)$,
\begin{equation*}
\begin{split}
\lim_{N\to \infty} \frac{Leb(B_r(0))}{N}\sum_{i=1}^N \phi(\xi_i)f(\xi_i) &= Leb(B_r(0)) \int_{\Omega'} \phi(\omega')f(\xi_1(\omega))P'(\ud \omega') \\
&= \int_{\Rd} \phi(x)f(x)1_{B_r(0)}(x)\ud x,\: P'-\textnormal{a.s.}
\end{split}
\end{equation*}
This also implies  that
\begin{equation*}
\lim_{N\to \infty} \frac{Leb(B_r(0))}{N}\sum_{i=1}^N \phi(\xi_i)\delta(\xi_i-\cdot)\overset{w}{\Longrightarrow} \phi(x)1_{B_r(0)}(x)\ud x,\: P'-\textnormal{ a.s.}
\end{equation*}
\end{lemma}
\begin{proof}
It is obvious that $\phi f \in L^1(\Rd)$ and the rest follows from the  law of large numbers.
\end{proof}

\begin{lemma}\label{lm:cont_dist}
For any $f\in \overset{\circ}{C} \mbox{}^{+}_b(\Rd), t>0$,
\begin{equation*}
E \left(1_{\{t<T(1)\}} \exp\left(-\int_{\Rd}\eta^{ac}_t(z)f(x)\ud x \right) \right)= E\left(1_{\{t<T(1)\}}\exp\left(-\int_{\Rd} X_t(\ud x) f(x)\right)\right).
\end{equation*}

\end{lemma}
\begin{proof} 
 We augment our probability space $(\Omega,\mathcal{F},P(\ud \omega))$ by taking the  Cartesian product with another probability space
 $(\Omega',\mathcal{F}',P'(\ud \omega'))$:
 \begin{equation*}
 (\tilde{\Omega},\tilde{\mathcal{F}},\tilde{P}) \equiv (\Omega\times\Omega',\mathcal{F}\times\mathcal{F}',P(\ud \omega)P'(\omega)).
 \end{equation*}
 We also denote expectations on these spaces by $E$,  $E'$ and $\tilde{E}$ respectively.
 Let $C^{++}_b(\Rd)$ denote  the space of bounded continuous functions on $\Rd$ such that for any $f\in  C^{++}_b(\Rd)$, we have $\inf_{x\in \Rd} f(x)>0$. 
  Let us fix  an arbitrary $f \in C^{++}_b(\Rd)$  and a  positive integer $n$.
 \par
 By the Borel theorem (see~\cite{Kal2}, Thm 3.19, p. 55) , for each $n\geq 1$ we can build on the probability space
 $(\Omega',\mathcal{F}',P'(\ud \omega'))$ a sequence $\{\xi_i^n(\omega)\}_{i=1}^\infty$ of i.i.d.
 random variables with the  density function
 \begin{equation*}
 g_{\xi}^n(x) = \left\{
 \begin{array}{ll}
 1/Leb(B_n(0)) & \text{if } |x| \leq n\\
 0  &\text{elsewhere. }
 \end{array} \right.
 \end{equation*}
 By  Lemma~\ref{lm:ll_conv_as} we get, that the equality
  \begin{align*}
& E\left(1_{\{t<T(1)\}} \exp\bigg(-\frac{Leb(B_n(0))}{N}\sum_{i=1}^N f(z_i)\eta^{ac}(z_i)\bigg)\right) \\
&\;\;\;= \exp\bigg(- \bigg<X_0, V_t\bigg(\frac{Leb(B_n(0))}{N}\sum_{i=1}^N f(z_i) \delta(z_i- \cdot)\bigg)\bigg>\bigg)
\end{align*}
holds for  Lebesgue   almost every  $\{z_i\}_{i=1}^N$ in $\mathbb{R}^d$.
By changing  $\{z_i\}_{i=1}^N$ to $\{\xi_i^n\}_{i=1}^N$, we obtain
 \begin{equation}\label{eq:lm_cont_prob_1}
 \begin{split}
& E\left(1_{\{t<T(1)\}} \exp\bigg(-\frac{Leb(B_n(0))}{N}\sum_{i=1}^N f(\xi_i)1_{B_n(0)}(\xi_i^n)\eta^{ac}(\xi_i^n)\bigg)\right)\\
&\;\;\; = \exp\bigg(- \bigg<X_0, V_t\bigg(\frac{Leb(B_n(0))}{N}\sum_{i=1}^N f(\xi_i)1_{B_n(0)}(\xi_i^n) \delta(\xi_i^n- \cdot)\bigg)\bigg>\bigg),\: P'- \mathrm{a. s.}
\end{split}
\end{equation}
By taking limits $N\to \infty$ on both sides of \eqref{eq:lm_cont_prob_1}, as well as  using Corollary  \ref{cor:conv_bounded_mesure}
and Lemma~\ref{lm:as_conv} we get the  equality
\begin{equation*}
\begin{split}
&E\left(1_{\{t<T(1)\}} \exp\bigg(-\int_{\Rd}\eta_t^{ac}(x)f(x)1_{B_n(0)}(x)\ud x \bigg)\right) \\
&\;\;\;= \exp\bigg(- \bigg<X_0, V_t(f 1_{B_n(0)})\bigg>\bigg),\: P' -\text{a.s.}.
\end{split}
\end{equation*}
Since both sides of the above equation are constants, we can drop $P'$-a.s., and get 
\begin{equation}\label{eq:abs_cont_eq}
E\left(1_{\{t<T(1)\}} \exp\bigg(-\int_{\Rd}\eta^{ac}(x)f)1_{B_n(0)}(x)\ud x \bigg)\right) = \exp\bigg(- \bigg<X_0, V_t(f 1_{B_n(0)})\bigg>\bigg).
\end{equation}
By Theorem $2.8$ in~\cite{AgEs},
\begin{equation*}
V_t(f 1_{B_n(0)})\leq V_t(f 1_{B_{n+1}(0)})
\end{equation*}
and
\begin{equation}\label{eq:abs_cont_AgEs}
\lim_{n\to \infty} V_t(f 1_{B_n(0)} ) = V_t(f).
\end{equation}

Now we take limits, as  $n\to \infty$ on both sides of \eqref{eq:abs_cont_eq}, use the monotone convergence theorem and \eqref{eq:abs_cont_AgEs} to  get
\begin{equation}\label{eq:abs_cont_dist}
E \left( 1_{\{t<T(1)\}} \exp\bigg(-\int_{\Rd}\eta^{ac}_t(z)f(x)\ud x \bigg)\right)  = \exp\left(- \left \langle X_0, V_t(f)\right\rangle\right).
\end{equation}
Since any function in $\overset{\circ}{C} \mbox{}^{+}_b(\Rd)$  can be approximated boundedly pointwise by functions from $C^{++}_b(\Rd)$, we can again  apply the  dominated convergence theorem and obtain that the equality \eqref{eq:abs_cont_dist}
holds  for any $f\in \overset{\circ}{C} \mbox{}^{+}_b(\Rd)$.
Recall,  that  
\begin{equation*}
E\left(1_{\{t<T(1)\}} \exp\left( -\left\langle X_t,f\right\rangle\right)\right) = \exp\left( -\left\langle V_t(f),X_0\right\rangle\right),
\end{equation*}
and we are done.
\end{proof}
Now we are ready to conclude the  proof of the main result.
\begin{proof} [Proof of Theorem \ref{thm:abs_cont}]
Fix arbitrary $t>0$.
By Corollary \ref{cor:expl_time},  Lemma~\ref{lm:cont_dist}, for every $f\in \overset{\circ}{C} \mbox{}^{+}_b(\Rd)$,
\begin{equation}\label{eq:thm_abs_cont_dist}
\begin{split}
&E\left(1_{\{t<T(1)\}}\exp\left( - \left\langle X_t,f\right\rangle\right)\right)\\
&\quad\quad= E \left( 1_{\{t<T(1)\}}\exp\left( - \int_{\Rd} \eta^{ac}_t(x)f(x)\ud x \right)\right).
\end{split}
\end{equation}
This equation implies,  that, on the event $\left\{t<T(1)\right\}$,
\begin{equation*}\label{eq:proof_abs_cont_1}
\int_{\Rd} X_t(\ud x) f(x) \eqdist \int_{\Rd} \eta^{ac}_t(x) f(x)\ud x,
\end{equation*}
where $\eqdist$ means equality in distribution.
By Lemma~\ref{lm:radon_nikodym} and the definition of $\eta_t^{ac}$, $\eta_t^{ac}$ is a version of the Radon-Nikodym derivative of $X_t(\ud x)$ on $\left\{t<T(1)\right\}$. Therefore, on $\left\{t<T(1)\right\}$,
\begin{equation}\label{eq:main_thm_eq_dist}
\int_{\Rd} X_t(\ud x) f(x) \geq \int_{\Rd} \eta^{ac}_t(x) f(x)\ud x,\: P-\mathrm{a.s.}
\end{equation}
Equations \eqref{eq:main_thm_eq_dist} and \eqref{eq:proof_abs_cont_1}  imply that

\begin{equation*}
\int_{\Rd} X_t(\ud x) f(x) = \int_{\Rd} \eta^{ac}_t(x) f(x)\ud x,\: P-\mathrm{a.s.}\: \textnormal{on}\: \left\{t<T(1)\right\}.
\end{equation*}
Since  $f\in \overset{\circ}{C} \mbox{}^{+}_b(\Rd)$ was  arbitrary,  this  completes  the  proof of the theorem.
\end{proof}

\section{Proof of Theorem~\ref{thm:semi}}
\label{chap:semi}
Many steps in the proof follow the lines from~\cite{AgEs}. However since the initial conditions are measures, modifications are required.

\subsection{Existence of solutions}

We now  prove the  existence of a  solution to  equation
\eqref{eq:ll_eq}
by the  Picard iterations.
Let $\mu\in  \MFORd$, and
\begin{equation}\label{eq:gen_non_eq}
w(x,t,\mu) = S_t\mu +\int_0^t (S_{t-s}\Psi (w(s,\cdot,\mu)) )(x)\ud s, \:  x \in\Rd,\: t>0,
\end{equation}
be an integral evolution equation such that $\Psi$ is some non-negative function defined on $\mathbb{R}_+$.
Recall that the Picard iterations for this equation are  defined by induction as follows:
\begin{equation}\label{eq:def_picard_iters}
\begin{split}
w_1(x,t,\mu) &= (S_t\mu)(x),\\
w_{n+1}(x,t,\mu) &= (S_t\mu)(x) + \int_0^t (S_{t-s}\Psi(w_n(s,\cdot,\mu)))(x)\ud s,\:x \in \Rd,\: t>0\\
&\hspace{ 2.5 cm} n=1,2,\ldots
\end{split}
\end{equation}
Notice that  \eqref{eq:ll_eq} is a particular case of \eqref{eq:gen_non_eq} with $\Psi(\lambda) = \lambda^\gamma$.
It is obvious that for $0<t<\infty$, the Picard iterations \eqref{eq:def_picard_iters} form the non-decreasing sequence: $ w_n(x,t,\mu)
\leq w_{n+1}(x,t,\mu)$.
In the next lemma we derive some properties  of the  Picard iterations.
\begin{lemma}\label{lm:picard_iters}
Let $\{v_n(x,t,\mu)\}_{n=1}^\infty$ be a sequence of Piccard iterations corresponding to  \eqref{eq:ll_eq}.
Then for every $n = 1,2,\ldots$ and any $0<t<\infty, x\in \Rd$,  the  following inequalities hold:
\begin{itemize}
\item[$(1)$] $0\leq v_n(t,x,\mu)\leq e^t (S_t\mu)(x) + e^t$,
\item[$(2)$]  $v^\gamma_n(t,x,\mu) \leq e^t (S_t\mu)(x) + e^t$.
\end{itemize}
\end{lemma}

\begin{proof}
First note that $v_n, n\geq 1,$ are non-negative by construction.
Let $\mu \in \MFRd$ and let us consider  a linear integral equation
\begin{equation}\label{eq:linear_eq}
u(t,x,\mu) = (S_t\mu)(x) + \int_0^t \big(S_{t-s}( u(s,\cdot,\mu)+1) \big)(x)\ud s,\: x \in \Rd, t>0.
\end{equation}
Let $\{u_n(x,t,\mu)\}_{n=1}^\infty$ be  corresponding Picard iterations. 
Note  that  \eqref{eq:linear_eq}  is a particular case of   \eqref{eq:gen_non_eq} with $\Psi(\lambda)=\lambda+1$.
Since    $\lambda^\gamma \leq  \lambda + 1$ for $\gamma \in (0,1)$,  one can easily see that $v_n(t,x,\mu) \leq u_n(t,x,\mu)$,  for all
$n\geq 1$. \par
On the other hand  by direct calculations, one gets that  for any $t \in (0,\infty)$
\begin{equation*}
\lim_{n\to \infty} u_n(x,t,\mu) \nearrow e^t(S_t\mu)(x) + e^t -1,\quad \text{as $n\to\infty$,}
\end{equation*}
and the first inequality of the lemma  follows.
The second inequality is a consequence of the first one   and the inequality $\lambda^\gamma \leq  \lambda + 1$.
\end{proof}
\begin{proposition}[Existence]\label{prop:existence}
Let $\mu \in \MFORd$. Then the
integral equation \eqref{eq:ll_eq}
has a  solution $v(\cdot,\cdot)$ which is a limit of Picard iterations and, for any $\phi\in L^1(\Rd)\cap C_b(\Rd)$
\begin{equation}\label{eq:ll_eq_init_cond}
\lim_{t\to 0} \int_{\Rd} v(t,x)\phi(x)\ud x = \int_{\Rd} \phi(x)\mu(\ud x).
\end{equation}
\par
Moreover $v(\cdot,\cdot)$ satisfies the following inequalities for any $(t,x) \in (0,\infty)\times \Rd$:
\begin{align}
v(t,x)&\leq e^t (S_t\mu)(x) + e^t,\label{eq:sol_ineq11}\\
v^\gamma(t,x)&\leq e^t (S_t\mu)(x) + e^t.\label{eq:sol_ineq12}
\end{align}
\end{proposition}

\begin{proof}
Let  $\{v_n(t,x)\}_{n=1}^\infty$ be a sequence  of Picard iterations corresponding  to
equation \eqref{eq:ll_eq}. By the  previous discussion for any $t \in (0,\infty)$, $\{v_n(t,\cdot)\}_{n=1}^\infty$  form a non-decreasing sequence and by Lemma~\ref{lm:picard_iters} we have
\begin{equation}\label{eq:pic_conv}
v_n(t,x)\leq e^t (S_t\mu)(x) + e^t.
\end{equation}
Lemma ~\ref{lm:simp_ineq} tells us   that  for every $t>0$, $(S_t\mu)(\cdot)$ is bounded.
Thus,  for any  $(t,x)$ in $(0,\infty) \times \Rd$, the sequence $\{v_n(t,x)\}_{n=1}^\infty$ is non-decreasing and bounded. Consequently there exists a bounded   limit $v(t,x)= \lim_{n\to\infty}v_n(t,x)$.
Inequalities \eqref{eq:sol_ineq11} and \eqref{eq:sol_ineq12} follow from existence of the limit and \eqref{eq:pic_conv}.

Now consider the sequence  of equations which defines the   Picard iterations:
\begin{equation}\label{eq_pic_iter2}
\begin{split}
&v_{l+1}(x,t) = \int_{\mathbb{R}^d} p_{t-s}(x-y)\mu(\ud y) + \int_0^t \int_{\mathbb{R}^d} p_t(x-y)v^\gamma_l(s,y)\ud y \ud s,\\
 &\hspace { 6 cm} l =1,2,\ldots.
\end{split}
\end{equation}
We have already proved that  the left  side of \eqref{eq_pic_iter2} converges boundedly pointwise  to $v(t,x)$. From the monotone convergence theorem it follows that the  right side converges to
\begin{equation*}
\int_{\mathbb{R}^d} p_{t-s}(x-y)\mu(\ud y) + \int_{0}^t \int_{\mathbb{R}^d} p_t(x-y)v^\gamma(s,y))\ud y \ud s.
\end{equation*}
Thus $v(t,x)$ satisfies equation \eqref{eq:ll_eq} for all $x \in\Rd,\: t>0$.\par
Now let us   verify \eqref{eq:ll_eq_init_cond}.
In the following discussion we can assume without loss of generality   that  $\phi\in C^+_b(\Rd)\cap L^1(\Rd) $.
Since the   family of functions  $\{p_t(\cdot)\}_{t>0}$ builds up the  \emph{Dirac family} (see~\cite[pp. 284-287, 348]{Lang}),
we can easily conclude that
\begin{equation}\label{eq:pic_2}
\begin{split}
\lim_{t\searrow 0 } \langle S_t\mu,\phi \rangle
&= \langle \mu,\phi \rangle.
\end{split}
\end{equation}
Now let us  prove
\begin{equation}\label{eq:picard_iter_3}
\lim_{t\to 0} \int_{\Rd}\left(\int_0^t \left(S_{t-s} v^\gamma(s,\cdot))(x) \right)\ud s \right)\phi(x)\ud x
= 0.
\end{equation}

First, using  the inequality \eqref{eq:sol_ineq11} we obtain
\begin{equation*}
\begin{split}
 \int_0^t (S_{t-s}v^\gamma(s))(x)\ud s &\leq  \int_0^t \big(S_{t-s}\left(
 e^{s} S_s\mu + e^{s}\right)\ud s \\
&=  (e^{t}-1) (S_t\mu)(x)+ (e^{t}-1).
\end{split}
\end{equation*}
Using the above bound we get:
\begin{equation*}
\begin{split}
&\lim_{t\searrow 0} \int_{\Rd} \left(\int_0^t(S_{t-s} v^\gamma(s))(x)\ud s\right)\phi(x)\ud x\\
&\leq  \lim_{t\searrow 0} \int_{\Rd} \left( (e^{t}-1) (S_t\mu)(x)+ (e^{t}-1)\right)\phi(x)\ud x\\
&=  \lim_{t\searrow 0} (e^t-1)(\mu(\Rd)\left\|\phi \right\|_\infty+ \|\phi\|_1)=0,
\end{split}
\end{equation*}
and \eqref{eq:picard_iter_3} follows. Equations  \eqref{eq:picard_iter_3} and \eqref{eq:pic_2} imply \eqref{eq:ll_eq_init_cond} and this completes  the proof of the proposition.
\end{proof}

\begin{corollary}
\label{cor:sol_inLinfsp}
	Let $\mu \in \MFORd$ and let $v(\cdot,\cdot,\mu)$ be a solution
	of  \eqref{eq:ll_eq} obtained as a limit of the Picard iterations in Proposition \ref{prop:existence}.
	Then  $v(\cdot,\cdot,\mu) \in L_{loc} ^\infty\left(((0,\infty) , L^{1,w}_+(\mathbb{R}^d)\right)$.
\end{corollary}
\begin{proof}
	Let $v(\cdot,\cdot,\mu)$ be a  solution constructed in Proposition~\ref{prop:existence}. Using the bound~\eqref{eq:sol_ineq11}, it is 
	easy to derive the result by standard Gaussian bounds.
\end{proof}

\subsection{Continuity of solutions}
\label{subsec:cont}
In this section we will prove the continuity of the solution obtained in Proposition \ref{prop:existence}. 

We start with the technical lemma, whose proof is 
 pretty standard, and therefore it is  omitted.  
\begin{lemma}\label{cor:dominating_func}
Fix $0<T_1<T_2$ and $r>0$.
Let $\{p_s(\cdot + z),\: s \in \left[T_1,T_2\right],\:|z|\leq r\}$ be a family of functions, where $p_s(\cdot)$ is a standard Gaussian kernel on $\Rd$.
Then, there exists a constant $K$, such that
\begin{equation*}
p_s(x+z) \leq  Kp_{2T_2}(x),\: \forall s \in [T_1,T_2],\: |z|\leq r, x\in \Rd. 
\end{equation*}

\end{lemma}
Now we are ready to state and prove the main proposition of Section~\ref{subsec:cont}.
\begin{proposition}[Continuity]\label{prop:continuity}
Let $\mu \in \MFORd$ and
let  $v(\cdot,\cdot)$ be a solution of \eqref{eq:ll_eq}
obtained as   a limit of Picard iterations in Proposition  \ref{prop:existence}. Then $v(\cdot,\cdot)  \in C^+((0,\infty)\times \Rd)$.
\end{proposition}
\begin{proof}
By construction the solution is clearly non-negative.
Now, let us  fix a point $(t,x)\in (0,\infty)\times \Rd$ and an arbitrary $\epsilon >0$. Let $\delta_i>0, i = 1,2,3,\:\, \delta_1<\delta_3  <t/10$.
In what follows we will show that $\delta_1$, $\delta_2$ and $\delta_3$ can be chosen sufficiently small so that  if  $ |\Delta t| <\delta_1$ and $|\Delta x|<\delta_2$ then
\begin{align}\label{eq:cont_thm}
\left| v(t+\Delta t, x +\Delta x ) - v(t,x)\right|\leq \epsilon.
\end{align}
We will bound  absolute value  of difference $v(t+\Delta t, x +\Delta x ) - v(t,x)$ only for the case of   $\Delta t\geq 0$, since   the case of $\Delta t<0$ can be treated similarly.
We split the difference $v(t+\Delta t, x +\Delta x ) - v(t,x)$ as follows:
\begin{align*}
v(t+\Delta t,x+\Delta x)-v(x,t) &= I(\Delta t, \Delta x) + J_1(\Delta t, \Delta x)- J_2(\Delta t, \Delta x) + J_3(\Delta t, \Delta x)\\
& +J_4(\Delta t, \Delta x)-J_5(\Delta t, \Delta x).
\end{align*}
Here
\begin{align*}
I(\Delta t,\Delta x) &= (S_{t+\Delta t }\mu)(x+\Delta x)- (S_{t}\mu)(x),\\
J_1(\Delta t,\Delta x) &= \intfour{t-\delta_3}{t+\Delta t}{t+\Delta t -s }{x+\Delta x},\\
J_2(\Delta t,\Delta x) &= \intfour{t-\delta_3}{t}{t -s }{x},\\
J_3(\Delta t,\Delta x) &= \intfour{\delta_3}{t-\delta_3}{t+\Delta t -s }{x+\Delta x}- \intfour{\delta_3 }{t-\delta_3}{t -s }{x},\\
J_4(\Delta t,\Delta x) &= \intfour{0}{\delta_3}{t+\Delta t -s}{x+\Delta x},\\
J_5(\Delta t,\Delta x) &= \intfour{0}{\delta_3}{t -s}{x}.
\end{align*}
Note that the integrals $J_1,J_2,J_4$ and $J_5$ have the same form:
\begin{equation}\label{eq:cont_thm_int}
J_{*} = \intfour{t_1}{t_2}{t_3 -s }{z},
\end{equation}
for appropriate $t_1,t_2,t_3\geq 0$ and $z\in \Rd$.
From the definitions of $\delta_1$, $\delta_3$  and $\Delta t$, it  follows that   $t_1,t_2$ and $t_3$ in \eqref{eq:cont_thm_int}  can vary but satisfy the inequalities $t_1<t_2$, $ t\leq t_3 $ and $t_2- t_1\leq 2 \delta_3$ hold.
Let us  bound $J_*$ from above. By Lemma~\ref{lm:simp_ineq} and Proposion \ref{prop:existence},
we easily get
\begin{equation}\label{eq:cont_gen_ineq}
\begin{split}
J_{*}& = \intfour{t_1}{t_2}{t_3 -s }{z}\\
&\leq \int_{t_1}^{t_2} \left(S_{t_3-s} ( e^s ( S_s\mu + 1))\right)(z) \ud s\\
&= \int_{t_1}^{t_2} e^s\left(\left(S_{t_3}\mu\right)(z) + 1\right)\ud s \\
&= (e^{t_2}-e^{t_1})\left(\left(S_{t_3}\mu\right)(z) + 1\right)\\
&\leq \frac{(e^{t_2}-e^{t_1})\left(\mu(\Rd)+1\right)}{(2\pi t_3)^{d/2}}\\
&\leq \frac{(e^{t_2}-e^{t_1})\left(\mu(\Rd)+1\right)}{(2\pi t)^{d/2}}
\end{split}
\end{equation}
where the last inequality follows from $t\leq t_3$. Recall that $t_2-t_1\leq 2\delta_3$,  and so by \eqref{eq:cont_gen_ineq}
we can choose $\delta_3$ sufficiently small so that, for $i =1,2,4,5$
\begin{align}\label{eq:cont_0}
J_i (\Delta t,\Delta x) \leq \epsilon/10,\: \Delta t<\delta_1 <\delta_3.
\end{align}
Let us  fix such  $\delta_3$. Let us recall    that $\Delta t<\delta_1<\delta_3$.
Now we will handle $J_3(\Delta t,\Delta x)$.
Write $J_3$ as $J_3(\Delta t,\Delta x)=J_{31}(\Delta t,\Delta x) - J_{32}$, where
\begin{align}
J_{31}(\Delta t,\Delta x)& = \int_{\delta_3}^{t-\delta_3} \int_{\Rd}p_{t+\Delta t-s }(x+\Delta x - y)v^\gamma(s,y)\ud y\ud s,\label{eq:int_j31}\\
J_{32}&= \int_{\delta_3}^{t-\delta_3}\int_{\Rd} p_{t-s}(x-y)v^\gamma(s,y)\ud y \ud s.
\end{align}
By Lemma~\ref{cor:dominating_func} and Proposition \ref{prop:existence},  we immediately get that there exists $K=K(\delta_1,\delta_1,\delta_3,t)$ such that
\begin{align*}
p_{t+\Delta t-s }(x+\Delta x - y)v^\gamma(s,y) &\leq Kp_{2t}(x-y)e^s((S_s\mu)(y)+1)\\
& \forall \Delta t \in (0,\delta_1) , s \in (\delta_3,t-\delta_3), |\Delta x|<\delta_2.
\end{align*}

It is easy to verify that
\begin{align*}
\int_{\delta_3}^{t-\delta_3} \int_{\Rd} K p_{2t}(x-y) e^s \left(\left(S_s\mu\right)(y)+1\right)\ud y \ud s <\infty.
\end{align*}
Therefore we can use the dominated convergence theorem and obtain:
\begin{equation}\label{eq:cont_1}
\begin{split}
\lim_{\substack{\Delta t \rightarrow 0 \\ \Delta x \rightarrow 0 }} J_{31}(\Delta t, \Delta x )&=\lim_{\substack{\Delta t \longrightarrow 0 \\ \Delta x \longrightarrow 0 }}
\int_{\delta_3}^{t-\delta_3} \int_{\Rd}p_{t+\Delta t -s }(x+\Delta x -y )v^\gamma(s,y)\ud y\ud s\\
&= \int_{\delta_3}^{t-\delta_3} \int_{\Rd}p_{t -s }(x -y )v^\gamma(s,y)\ud y\ud s\\
&= J_{32}.
\end{split}
\end{equation}
Similarly we show
\begin{equation}\label{eq:cont_2}
\lim_{\substack{\Delta t \longrightarrow 0 \\ \Delta x \longrightarrow 0 }}
(S_{t+\Delta t}\mu)(x+\Delta x) = (S_{t}\mu)(x),
\end{equation}
and thus from~\eqref{eq:cont_1} , \eqref{eq:cont_2} and the definition of $I(\Delta t, \Delta x ), J_{3}(\Delta t, \Delta x )$ we get 
\begin{equation}
\lim_{\substack{\Delta t \longrightarrow 0 \\ \Delta x \longrightarrow 0 }}I(\Delta t, \Delta x )+J_{3}(\Delta t, \Delta x )=0. 
\end{equation}
This implies that there exist $\delta_1,\delta_2\in (0,\delta_3)$ sufficiently small such that 
  for $|\Delta t| <\delta_1$ and $|\Delta x|<\delta_2$
\begin{align}\label{eq:cont_3}
\left| J_3(\Delta t,\Delta x) \right| + \left| J_3(\Delta t,\Delta x) \right| \leq \epsilon/2.
\end{align}
Thus  we get from  \eqref{eq:cont_0} and \eqref{eq:cont_3}  that
\begin{align*}
&\left|v(t+\Delta t, x+\Delta x) - v(t,x)\right| \\
&= | I(t+\Delta t, x+\Delta x) + J_1(t+\Delta t, x+\Delta x)- J_2(t+\Delta t, x+\Delta x) \\
&\quad\: + J_3(t+\Delta t, x+\Delta x) + J_4(t+\Delta t, x+\Delta x) -J_5(t+\Delta t, x+\Delta x) |\\
&\leq \left|I(t+\Delta t, x+\Delta x) \right| + \left|J_1(t+\Delta t, x+\Delta x) \right| + \left|J_2(t+\Delta t, x+\Delta x) \right| \\
& \quad\:  + \left|J_3(t+\Delta t, x+\Delta x) \right| + \left|J_4(t+\Delta t, x+\Delta x) \right| + \left|J_5(t+\Delta t, x+\Delta x) \right| \\
&\leq \epsilon,\: \forall \Delta t ,\Delta x :\: \Delta t \in (0, \delta_1), |\Delta x| < \delta_2.
\end{align*}
Since $\epsilon >0$ was arbitrary we are done.
\end{proof}

\subsection{Uniqueness of solutions }

The proof of uniqueness is again  based on proofs in~\cite{AgEs} which are adjusted to our case.
Let us recall that 
$\gamma'= 1/(1-\gamma).$

We start with the technical lemma, whose proof is 
 pretty standard, and therefore it is  omitted.  
\begin{lemma}\label{lm:heat_semi_action_on_exp}
Let $\mu \in \MFORd$. Then, for every $t>0$, there exist two positive constants
$c(t)$, $a(t)$ such that following inequality holds:
\begin{equation*}
(S_t\mu)(x)  \geq c(t)\cdot e^{-a(t)|x|^2},\: \forall x \in \Rd.
\end{equation*}
\end{lemma}

In the next lemma we proof an important lower bound. 
\begin{lemma}\label{lm:uniq_inequality} 
Let $\mu \in \MFORd$  and $u(t,x)$ be a non-negative function on $(0,\infty)\times\mathbb{R}^d$ such that, for any $t\in (0,\infty)$ and any $x\in\mathbb{R}^d:$
\begin{displaymath}
v(t,x) \geq S_t\mu + \int_0^t \left(S_{t-s} v^\gamma ( s) \right)(x)\ud s .
\end{displaymath}
Then
\begin{align}\label{lm_uniq_0}
v(t,x) > \left(\left(1-\gamma\right)t\right)^{\gamma'}, \: \forall (t,x) \in (0,\infty)\times \Rd.
\end{align}
\end{lemma}
\begin{proof}
Let us fix an arbitrary $t_0>0$ and define
\begin{align*}
\tilde{v}(t)\equiv v(t+t_0),\: \forall t \geq 0.
\end{align*}
Using this definition one can easily check that  
\begin{align*}
\tilde{v}(t) 
&\geq S_t \tilde v_0 + \int_0^t S_{t-s} \tilde{v}^\gamma(s)\ud s, \quad t\geq 0,
\end{align*}
where $\tilde{v}_0 =  \tilde v(0) \geq S_{t_0}\mu$, and  the last inequality follows by definition of $\tilde v$ and assumptions on $v$. Thus, by Lemma~\ref{lm:heat_semi_action_on_exp} there exists  $a(t_0)>0$ and $c(t_0)>0$ such that
\begin{align*}
\tilde{v}_0(x) \geq \left(S_{t_0}\mu\right)(x) \geq c(t_0) e^{-a(t_0)|x|^2},\: \forall x \in \Rd.
\end{align*}
By Lemma~$2.2$ in~\cite{AgEs}  we have
\begin{equation*}
\tilde{v}(t,x) \geq ((1-\gamma)t)^{\gamma'}, \forall (t,x)\in (0,\infty)\times \Rd.
\end{equation*}
 Since $t_0 \geq 0 $ was arbitrary, we get
\begin{align*}
v(t,x) \geq ((1-\gamma)t)^{\gamma'}, \forall (t,x)\in (0,\infty)\times \Rd,
\end{align*}
and we are done. 
\end{proof}

\begin{lemma}[Comparison lemma]\label{lm:comp}
Let
\begin{displaymath}
v,u\in L_{loc} ^\infty\left((0,\infty) , L^{1,w}(\mathbb{R}^d)\right)\cap C((0,\infty)\times \Rd)
\end{displaymath}
 be non-negative functions such that, for all $t>0$,
\begin{equation*}
\begin{split}
u(t) \geq S_t\nu + \int_0^t S_{t-s} u^\gamma(s) \ud s,  \\
v(t) \leq S_t\mu + \int_0^t S_{t-s} v^\gamma(s) \ud s .
\end{split}
\end{equation*}
Here $\mu$, $\nu \in \MFORd$  are such that
\begin{equation*}
\nu(f) \geq \mu(f) ,\forall f  \in C^+_b(\mathbb{R}^d).
\end{equation*}
Then
\begin{equation*}
u(t,x) \geq v(t,x), \:\:\textnormal{for all}\: (t,x) \in (0,\infty)\times\mathbb{R}^d.
\end{equation*}
\end{lemma}

\begin{proof}
Define
\begin{equation*}
g(t) \equiv v(t) - u(t) .
\end{equation*}
We will now   prove that  $g_+(t)\equiv\max(g(t),0)= 0$.

Fix arbitrary $T>0$. We use the  condition $\nu \geq \mu$ and  an elementary   inequality $(a^\gamma - b^\gamma)\leq ((a-b)_+)^\gamma$  to get
\begin{align}\label{eq:comp_0}
g(t)& \leq S_t(\mu -\nu) + \int_0^t S_{t-s} (v^\gamma(s) - u^\gamma(s)) \ud s \notag \\
\leq & \int_0^t S_{t-s}(v^\gamma(s) - u^\gamma(s))_+ \ud s \notag \\
\leq & \int_0^t S_{t-s} \left(\left(g_+(s)\right)^\gamma \right)\ud s.
\end{align}
From this point the proof follows the proof of Theorem~2.8  in~\cite{AgEs}
while using Lemma~\ref{lm:uniq_inequality} whenever necessary. We left the details to the reader.
\end{proof}

The uniqueness for \eqref{eq:ll_eq} follows easily from the above  comparison Lemma~\ref{lm:comp}.  
\begin{proposition}[Uniqueness]\label{prop:uniqueness}
Let $\mu \in \MFORd$. There is at  most one  solution to
 \eqref{eq:ll_eq} which belongs to $ L_{loc} ^\infty\left([0,\infty) , L^{1,w}_+(\mathbb{R}^d)\right)\cap C^+((0,\infty)\times \Rd)$.
\end{proposition}
\begin{proof}
Suppose  there exist two functions $v,u\in L_{loc} ^\infty\left([0,\infty) , L^{1,w}_+(\mathbb{R}^d)\right)\cap C^+((0,\infty)\times \Rd)$ that solve equation \eqref{eq:ll_eq}  for the same initial measure $\mu$. Then by Lemma~\ref{lm:comp}, $v(t,x)\geq u(t,x)$ and $u(t,x) \geq v(t,x)$ for any $(t,x) \in (0,\infty)\times \Rd$, and thus $u=v$ and  we are done.
\end{proof}

\subsection{Continuous dependence of solutions on initial data}
In the previous sections we  proved the existence and uniqueness of solutions to   equation \eqref{eq:ll_eq}, or  looking from different perspective we proved  for every $(t,x)\in (0,\infty)\times \Rd$ the existence of the mapping
\begin{equation*}
v(t,x,\cdot) : \MFORd \to \mathbb{R}_{++}.
\end{equation*}
Here $v(t,x,\mu)$ is a solution to   equation \eqref{eq:ll_eq} with initial datum $\mu$.\newline
In this section we will prove the  continuity of this mapping.

\begin{lemma}\label{lm:concave_mapping}
For any $(t,x)\in (0,\infty)\times \Rd$, the  mapping
\begin{equation*}
v(t,x,\cdot) : \MFORd \mapsto \mathbb{R}_{++}
\end{equation*}
 is concave, i.e.
\begin{equation*}
\begin{split}
v(t,x,\lambda\mu + (1-\lambda)\nu) \geq &\lambda v(t,x,\mu) + (1-\lambda)v(t,x,\nu),\\
& \forall \lambda\in (0,1),\:\: \forall (t,x) \in (0,\infty)\times\Rd.
\end{split}
\end{equation*}
\end{lemma}
\begin{proof}
Since the function $x\to x^\gamma$ is concave,  for any positive $a,b$ and $\lambda \in (0,1)$, we have
\begin{equation}\label{eq:lm_cont_init}
\lambda a^\gamma + (1-\lambda) b^\gamma \leq ( \lambda a + (1-\lambda) b)^\gamma.
\end{equation}
\par
Let us fix an arbitrary $\lambda \in (0,1)$ and  define $u(t,z,\mu,\nu,\lambda)$ as follows
\begin{equation*}
 u(t,x,\mu,\nu,\lambda)\triangleq \lambda v(t,\mu) + (1-\lambda)v(t,\nu).
\end{equation*}
Then we have
\begin{equation*}
\begin{split}
u(t,x,\mu,\nu,\lambda) &= \lambda v(t,x,\mu) + (1-\lambda)v(t,x,\nu) \\
&= \big(S_t\sigma\big)(x) + \int_0^t \Big(S_{t-s} \Big( \lambda v^\gamma(s,\mu) + (1-\lambda) v^\gamma(s,\nu)\Big)\Big)(x)\ud s \\
& \leq \big(S_t\sigma\big)(x) + \int_0^t \Big( S_{t-s}\Big(\lambda v(s,\mu) + (1-\lambda) v(s,\nu)\Big)^\gamma\Big)(x) \ud s\\
& = \big(S_t\sigma\big)(x) + \int_0^t \Big(S_{t-s}( u^\gamma(s,\mu,nu,\lambda)\Big)(x) \ud s ,
\end{split}
\end{equation*}
where the above inequality  follows from \eqref{eq:lm_cont_init}, and we set
 $\sigma = \lambda \mu + (1-\lambda)\nu$. Hence we obtained
\begin{equation}\label{eq:lm_init_cont_1}
u(t,x)  \leq   \left(S_t\sigma\right)(x) + \int_0^t \left(S_{t-s} u^\gamma(s)\right)(x) \ud s .
\end{equation}
We now  recall that by definition
\begin{equation}\label{eq:lm_init_cont_2}
v(t,x,\sigma) = \left(S_t\sigma\right)(x) + \int_0^t \left(S_{t-s} v(s,\sigma)^\gamma \right)(x)\ud s .
\end{equation}
 and it is left to  use comparison  Lemma~\ref{lm:comp}.
\end{proof}
Before we prove continuity let us recall that   the  space of signed finite measures  on $\Rd$, $M_{F,S}(\Rd)$,  is a topological vector space with topology of weak convergence.
\begin{proposition}\label{prop:cont_init_data}
For any fixed $(t,x)\in (0,\infty)\times \Rd$, the  mapping
\begin{equation*}
v(t,x,\cdot): \MFORd \mapsto \mathbb{R}_{++}
\end{equation*}
is continuous.
\end{proposition}
\begin{remark}
It follows from  the above proposition  that the weak convergence of initial measures implies  pointwise convergence of solutions to  equation \eqref{eq:ll_eq}.
\end{remark}

\begin{proof}
By Lemma~\ref{lm:concave_mapping} for any $(t,x)\in (0,\infty)\times \Rd$, the  mapping $\mu\mapsto v(t,x,\cdot)$ is concave, and $v(t,x,\mu) \geq 0$ for any $\mu\in \MFORd$. Hence, by Lemma~$2.1$ in~\cite{ET},
mapping $\mu\mapsto v(t,x,\mu)$ is continuous.
\end{proof}
Now we are able to prove Theorem \ref{thm:semi}.
\begin{proof}[Proof of  Theorem \ref{thm:semi}]
The statement of the theorem follows from   Propositions \ref{prop:existence}, \ref{prop:continuity}, \ref{prop:uniqueness}, Corollary~\ref{cor:sol_inLinfsp} and Proposition~\ref{prop:cont_init_data}.
\end{proof}

%

\end{document}